\pdfoutput=1

\documentclass[12pt]{amsart}

\usepackage{amsmath} 
\usepackage{amsthm} 
\usepackage{amssymb} 

\usepackage{microtype} 
\usepackage{pinlabel} 

\usepackage{MnSymbol} 

\usepackage[scaled=0.9]{sourcecodepro} 

\usepackage[hidelinks, pagebackref]{hyperref}

\makeatletter
\define@key{href}{font}{#1}
\makeatother
\usepackage{xpatch}
\newcommand\hrefdefaultfont{\ttfamily}
\xpatchcmd\href{\setkeys{href}{#1}}{\setkeys{href}{font=\hrefdefaultfont,#1}}{}{\fail}

\renewcommand*{\backref}[1]{}
\renewcommand*{\backrefalt}[4]{
  \ifcase #1 
  [No citations.]
  \or [#2]
  \else [#2]
  \fi }

\let\originalleft\left
\let\originalright\right
\renewcommand{\left}{\mathopen{}\mathclose\bgroup\originalleft}
\renewcommand{\right}{\aftergroup\egroup\originalright}

\newcommand{\calA}{\mathcal{A}}

\newcommand{\calC}{\mathcal{C}}
\newcommand{\calD}{\mathcal{D}}

\newcommand{\calF}{\mathcal{F}}

\newcommand{\calP}{\mathcal{P}}

\newcommand{\calU}{\mathcal{U}}

\newcommand{\calX}{\mathcal{X}}

\newcommand{\RR}{\mathbb{R}}

\newcommand{\ZZ}{\mathbb{Z}}

\newcommand{\st}{\mathbin{\mid}} 
\newcommand{\from}{\colon} 
 
\newcommand{\homeo}{\mathrel{\cong}} 

\newcommand{\cross}{\times}

\newcommand{\closure}[1]{{\overline{#1}}}

\newcommand{\bdy}{\partial} 

\newcommand{\diam}{\operatorname{diam}} 

\newcommand{\MCG}{\operatorname{MCG}} 

\input{header_article.tex}

\theoremstyle{plain}
\newtheorem{XXXtheoremQED}[equation]{Theorem} 
  {\pushQED{\qed}\begin{XXXtheoremQED}}
  {\popQED\end{XXXtheoremQED}}

\newcommand{\fakeenv}{} 

\newenvironment{restate}[2]  
{ 
 \renewcommand{\fakeenv}{#2} 
 \theoremstyle{plain} 
 \newtheorem*{\fakeenv}{#1~\ref{#2}} 
 \begin{\fakeenv}
}
{
 \end{\fakeenv}
}

\newenvironment{restated}[2]  
{ 
 \renewcommand{\fakeenv}{#2} 
 \theoremstyle{definition} 
 \newtheorem*{\fakeenv}{#1~\ref{#2}} 
 \begin{\fakeenv}
}
{
 \end{\fakeenv}
}

\newcommand{\Cutoff}{\mathsf{K}}
\newcommand{\Coarse}{\mathsf{C}}
\newcommand{\Diameter}{\mathsf{D}}
\newcommand{\Link}{\mathsf{k}}
\newcommand{\Close}{\mathsf{L}}

\newcommand{\AC}{\mathcal{AC}}
\newcommand{\asydim}{\operatorname{\mathsf{dim}_{\mathsf{asym}}}}

\begin{document}

\title[Asymptotic dimensions]{Asymptotic dimensions of the arc graphs and disk graphs}

\author[Fujiwara]{Koji Fujiwara}
\address{\hskip-\parindent
        Department of Mathematics\\
        Kyoto University\\
        Kyoto 606-8502, Japan}
\email{kfujiwara@math.kyoto-u.ac.jp}

\author[Schleimer]{Saul Schleimer}
\address{\hskip-\parindent
        Mathematics Institute\\
        University of Warwick\\
        Coventry CV4 7AL, United Kingdom}
\email{s.schleimer@warwick.ac.uk}

\thanks{This work is in the public domain.
The first author is supported by  Grants-in-Aid
for Scientific Research from JSPS (20H00114, 15H05739)}

\date{\today}

\begin{abstract}
We give quadratic upper bounds for the asymptotic dimensions of the arc graphs and disk graphs.
\end{abstract}



\maketitle

\section{Introduction}
\label{Sec:Intro}

The asymptotic dimension, denoted $\asydim X$, of a metric space $X$ was introduced by Gromov~\cite[page~29]{g} as a large scale analog of the covering dimension.  
The curve graph, $\calC(S)$, for a surface $S = S_{g,b}$ was introduced by Harvey~\cite[page~246]{ha} as a sort of Bruhat-Tits building for Teichm\"uller space.  
It has since been generalised in many ways.
Bell and Fujiwara first proved that the asymptotic dimension of $\calC(S)$ is finite~\cite[Corollary~1]{bf}.  
More recently, Bestvina and Bromberg proved that $\asydim \calC(S_{g}) \leq 4g - 4$ 
(when $g > 1$) and that $\asydim \calC(S_{g,b}) \leq 4g - 3 + b = \xi'(S_{g,b})$ 
when $g > 0$ and $b > 0$ (or $g = 0$ and $b > 2$)~\cite[Corollary~1.1]{bb}.

Here we combine the machineries of~\cite{bbf} and~\cite{ms} to produce a quasi-isometric embedding of the arc graph $\calA(S, \Delta)$ into a finite product of quasi-trees of curve complexes.  
From this we deduce the following. 

\begin{restate}{Corollary}{Cor:ArcBound}
Suppose that $S = S_{g,b}$ has non-empty boundary.  
Suppose that $\Delta \subset \bdy S$ is a non-empty union of components.
Finally, suppose that $\xi'(S) \geq 1$.  
Then 
\[
\asydim \calA(S, \Delta) \leq \frac{(4g+b) (4g+b-3)}{2} - 2
\]
\end{restate}

\noindent
Sisto~\cite{Sisto22} suggests that the machineries of~\cite[Theorem~5.2]{Sisto17} and~\cite[Theorem~A.2]{Vokes22} can be combined to obtain a similar result.

We also obtain the following result for the disk graph $\calD(M, S)$ of a compression body.

\begin{restate}{Corollary}{Cor:DiskBound}
Suppose $M$ is a non-trivial spotless compression body with upper boundary $S = S_{g, b}$.
Suppose that $\xi'(S) \geq 1$.
Then 
\[
\asydim \calD(M, S) \leq \frac{(4g+b) (4g+b-3)}{2} - 2
\]
\end{restate}


\noindent
Hamenst\"adt~\cite[Theorem~3.6]{hii} has obtained a similar result when $M$ is a handlebody; 
see \refrem{Hamenstadt}.

Obtaining lower bounds better than linear, if even possible, would require new ideas.  
Therefore we end our introduction with the following.

\begin{question}
\label{Que:Tight}
How tight are the upper bounds of Corollaries~\ref{Cor:ArcBound} and~\ref{Cor:DiskBound}?
\end{question}

\section{Background}
\label{Sec:Background}

We now more-or-less follow the conventions of~\cite{bd}. 

Suppose that $(X, d_X)$ is a metric space.
Suppose that $U$ and $V$ are non-empty subsets of $X$ with bounded diameter.
Throughout the paper we use the following notational convention.
\[
d_X(U, V) = \diam_X(U \cup V)
\]

We write $p =_\Coarse q$ if, for non-negative numbers $p$, $q$, and $\Coarse$ we have both $q \leq \Coarse p + \Coarse$ and $p \leq \Coarse q + \Coarse$.  
Also, we use the \emph{cut-off} function: 
$[p]_\Coarse$ is equal to $p$ if $p \geq \Coarse$ and is zero otherwise. 

Suppose that $X$ and $Y$ are metric spaces. 
A relation $f \from X \to Y$ is a \emph{coarse map} if there is a constant $\Coarse$ so that, for all $x \in X$, the image $f(x)$ is non-empty and has $\diam_Y(f(x)) \leq \Coarse$. 
A coarse map $f \from X \to Y$ is a \emph{coarse embedding} if there are functions $F, G \from \RR \to \RR$ so that 
\begin{itemize}
\item
$\lim_{t \to \infty} F(t) = \lim_{t \to \infty} G(t) = \infty$ and, 
\item
for all $x, y \in X$, we have 
\[
F(d_X(x,y)) \leq d_Y(f(x), f(y)) \leq G(d_X(x,y))
\]
\end{itemize}
A coarse map $f \from X \to Y$ is \emph{coarsely onto} if there is a constant $\Coarse > 0$ so that for all $y \in Y$ there is a point $x \in X$ with $d_Y(f(x), y) < \Coarse$.  

A coarse map $f \from X \to Y$ is \emph{coarsely Lipschitz} if there is a constant $\Coarse > 0$ so that for all $x, y \in X$ we have 
\[
d_Y(f(x), f(y)) \leq \Coarse \cdot d_X(x,y) + \Coarse
\]
That is, we have no lower bound, but we require the upper bound to be affine.

A coarse embedding $f$ is a \emph{quasi-isometric embedding} if the functions $F$ and $G$ are both affine (with positive coefficients).
The more usual definition is to require a constant $\Coarse > 0$ so that $d_X(x, y) =_\Coarse d_Y(f(x), f(y))$ for all $x, y \in X$.  
A quasi-isometric embedding $f$ is a \emph{quasi-isometry} if $f$ is also coarsely onto.

\subsection{Asymptotic dimension}

We now follow~\cite[Section~1.E]{g}. 
A metric space $X$ has \emph{asymptotic dimension} $\asydim(X)$ at most $n$ if for every $R > 0$ there is a $\Diameter > 0$ and a cover $\calU$ of $X$ so that 
\begin{itemize}
\item
for all $U \in \calU$ we have $\diam_X(U) \leq \Diameter$ and 
\item
every metric $R$-ball in $X$ intersects at most $n + 1$ sets in $\calU$.
\end{itemize}
For example, trees have asymptotic dimension at most one.
More generally, following~\cite[page~1272]{bd}, a collection of metric spaces $\calX$ has asymptotic dimension at most $n$, \emph{uniformly}, if for every $R$ there is a constant $\Diameter$ so that every $X \in \calX$ has a cover as above.

If two metric spaces $X$ and $Y$ are quasi-isometric and $\asydim(X) \leq n$, then $\asydim(Y) \leq n$.  In view of this, we say a finitely generated group $G$ has asymptotic dimension at most $n$ if some (and thus all) of its Cayley graphs have asymptotic dimension at most $n$.  

We list two well-known facts about asymptotic dimension. 

\begin{fact}\cite[Theorem~32]{bd}
\label{Fac:Product}
Suppose that $U$ and $V$ are metric spaces.
We give $U \times V$ the $\ell^1$ product metric.
Then 
\[
\asydim U \times V \leq \asydim U  + \asydim V \qedhere
\]
\end{fact}

\begin{fact}\cite[Proof of Proposition~22]{bd}
\label{Fac:Subspaces}
If $U$ coarsely embeds into $V$ then 
\[
\asydim U \leq \asydim V \qedhere
\]
\end{fact}

\subsection{Quasi-trees of metric graphs}
\label{Sec:BBF}




We quickly review the machinery of \emph{quasi-trees of metric spaces}, as introduced in~\cite[Section~4]{bbf}.
Suppose that $\calF$ is a collection of metric graphs.
Suppose also that we have, for every pair of distinct graphs $A, B \in \calF$, a non-empty subset $\pi_B(A) \subset B$.
Also, fix a sufficiently large constant $\Link > 0$. 

With respect to the data $(\calF, \pi, \Link)$ we require the following axioms. 

\begin{axiom}[Bounded projections]
\label{Ax:BoundedProjections}
For distinct $A, B \in \calF$ we have:
\[
\diam_B(\pi_B(A)) \leq \Link \qedhere
\]
\end{axiom}

For $A, B, C \in \calF$, and if $A \neq B$ and $B \neq C$, we adopt the following shorthand:
\[
d_B(A, C) = d_B(\pi_B(A), \pi_B(C))
\]

\begin{axiom}[Behrstock inequality]
\label{Ax:Behrstock}
For distinct $A, B, C \in \calF$ at most one of the following is greater than $\Link$: 
\[
d_A(B, C), \quad 
d_B(A, C), \quad 
d_C(A, B) \qedhere
\]
\end{axiom}

\begin{axiom}[Large links]
\label{Ax:LargeLink}
For distinct $A, C \in \calF$ the following set is finite: 
\[
\left\{ B \in \calF \st \mbox{$A \neq B$, $B \neq C$, and $d_B(A, C) > \Link$} \right\} \qedhere
\]
\end{axiom}

\noindent 
We call these the \emph{BBF axioms}.
These are called (P0), (P1), and (P2) in~\cite{bbf}.

Suppose that the data $(\calF, \pi, \Link)$ satisfies the BBF axioms.
Then, by~\cite[Theorem~A]{bbf}, for every sufficiently large $\Cutoff \geq \Link$ there is a metric graph $\calC(\calF) = \calC_\Cutoff(\calF)$ called the \emph{quasi-tree of graphs}.  
We denote the metric on $\calC(\calF)$ by $d_\calC$.  
We now list several properties of $\calC(\calF)$. 

\begin{fact}
\label{Fac:Embed}\cite[Theorem~A and Definition~3.6]{bbf}.
By construction, every $A \in \calF$ isometrically embeds into $\calC(\calF)$.
The image is totally geodesic.
For distinct $A, B \in \calF$ their images in $\calC(\calF)$ are disjoint.
Finally, these images cover the vertices of $\calC(\calF)$.
\end{fact}

Thus we may identify the vertices of $A$ with their images in $\calC(\calF)$.
We extend the definition of $\pi_B$ as follows.
If $b \in B$ then we take $\pi_B(b) = b$.
If $a \in A$, and $A$ is distinct from $B$, then we take $\pi_B(a) = \pi_B(A)$.
One may think of $\pi_B \from \calC(\calF) \to B$ as being a ``closest points projection'' map.
We adopt the shorthand 
\[
d_B(a, c) = d_B(\pi_B(a), \pi_B(c))
\]

\begin{fact}\cite[Definition~4.1]{bbf}
\label{Fac:Close}
By construction, there is a constant $\Close$ (depending only on $\Cutoff$) with the following property.
Suppose that $A, B \in \calF$ are graphs so that the set of \refax{LargeLink} is empty.
Then $d_\calC(\pi_A(B), \pi_B(A)) \leq \Close$.
\end{fact}


We have the following \emph{distance estimate} for $d_\calC$.

\begin{theorem}\cite[Theorem~4.13]{bbf}
\label{Thm:BBFEstimate}
Suppose that $\calF$ is a family of metric graphs satisfying the BBF axioms.
Suppose that $\Cutoff$ is sufficiently large.
Suppose that $\calC_\Cutoff(\calF)$ is the quasi-tree of graphs.
Then every $\Cutoff'$, sufficiently larger than $\Cutoff$, has the follow property. 
For any $a, c \in \calC(\calF)$ we have
\[
\frac{1}{2} \sum[d_B(a,c)]_{\Cutoff'} 
      \leq d_{\calC(\calF)}(a,c) 
      \leq 6\Cutoff + 4 \sum [d_B(a,c)]_\Cutoff
\]
where both sums are taken over all $B \in \calF$. \qed
\end{theorem}

Suppose that $G$ is a group acting on $\calF$.
We further assume that, for any $g \in G$ and for any $A \in \calF$ there is an isometry $g_A \from A \to g \cdot A$.
We suppose that these isometries have the following consistency properties.
Suppose that $g, h \in G$ are group elements and $A, B \in \calF$ are graphs with $B = g \cdot A$. 
\begin{itemize}
\item
For all $a \in A$ we have $h_B (g_A (a)) = (hg)_A (a)$. 
\item
For any $C \in \calF$ we have $g_A(\pi_A(C)) = \pi_B(g \cdot C)$. 
\end{itemize}
From \cite[Section~3.7]{bbf} we deduce the following:
there is an isometric action of $G$ on the quasi-tree of graphs $\calC(\calF)$ which extends the action of the isometries $g_A$.

We also have the following control on the asymptotic dimension of $\calC(\calF)$.

\begin{theorem}\cite[Theorem~4.24]{bbf}
\label{Thm:BBFBound}
Suppose that $\calF$ is a family of metric graphs satisfying the BBF axioms. 
Suppose that $\calF$ has asymptotic dimension at most $D$, uniformly.
Then $\asydim(\calC(\calF)) \leq D + 1$. \qed
\end{theorem}

When all of the metric graphs are quasi-trees this can be improved.

\begin{theorem}\cite[Theorem~B(ii)]{bbf}
\label{Thm:BBFBoundTree}
Suppose that $\calF$ is a family of quasi-trees satisfying the BBF axioms.
Suppose futher that the quasi-isometry constants are uniformly bounded.
Then $\asydim(\calC(\calF)) \leq 1$. \qed
\end{theorem}

The notion of a \emph{quasi-tree of metric spaces} was introduced and used to prove \cite[Theorem~D]{bbf}: 
mapping class groups (of connected, compact, oriented surfaces) have finite asymptotic dimension.  

\subsection{Surfaces, curves, and arcs}

Let $S = S_{g,b}$ denote the connected, compact, oriented surface of genus $g$ with $b$ boundary components.  
The \emph{complexity} of $S$ is defined to be $\xi(S) = 3g - 3 + b$.  
This counts the number of curves in any pants decomposition of $S$.  
We will always assume that $\xi(S) \geq 1$.  
We will also need the \emph{modified complexity} $\xi'(S) = 4g - 3 + b$.  
If $S$ is closed then we will simply write $S_g$ for $S_{g,0}$. 

Suppose that $\alpha$ is an embedded arc or curve in $S$.  
We call the embedding \emph{proper} if $\alpha \cap \bdy S = \bdy \alpha$.
A properly embedded arc or curve $\alpha$ in $S$ is \emph{essential} if it does not cut a disk off of $S$.  
A properly embedded curve $\alpha$ is \emph{non-peripheral} if it does not cut an annulus off of $S$.

A \emph{proper isotopy} is an isotopy through proper embeddings.
Let $[\alpha]$ denote the proper isotopy class of $\alpha$.
Given $\alpha$ and $\beta$, properly embedded arcs or curves, we define their \emph{geometric intersection number}:
\[
i(\alpha, \beta) = \min \{ |\alpha' \cap \beta'| : \alpha' \in [\alpha], \beta' \in [\beta] \}
\]
Note that $i(\alpha, \beta) = 0$ if and only if they have disjoint (proper isotopy) representatives. 
To lighten the notation, we typically will not distinguish between a curve (or arc) $\alpha$ and its proper isotopy class $[\alpha]$. 

A connected, compact subsurface $X \subset S$ is \emph{essential} if every component of $\bdy X$ is either a component of $\bdy S$ or is essential and non-peripheral in $S$.  
If $X$ is essential we define the \emph{relative boundary} of $X$ to be $\bdy_S X = \bdy X - \bdy S$. 

\begin{remark}
\label{Rem:Nested}
Note that if $X \subset Y$ are both essential subsurfaces of $S$, then $\xi'(X) \leq \xi'(Y)$.  
Equality holds if and only if $X$ and $Y$ are isotopic. 
\end{remark}

We say that a properly embedded curve or an arc $\alpha$ \emph{cuts} $X$ if every $\alpha' \in [\alpha]$ intersects $X$. 
If $\alpha$ does not cut $X$ then we say that $\alpha$ \emph{misses} $X$.
Suppose that $X$ and $Y$ are essential, and non-isotopic, subsurfaces of $S$.  
We say that $X$ is \emph{nested} in $Y$ if it is (perhaps after an isotopy) contained in $Y$.  
We say that $X$ and $Y$ \emph{overlap} if $\bdy_S X$ cuts $Y$ and $\bdy_S Y$ cuts $X$. 

\subsection{Curve and arc graphs}

We now define the \emph{curve graph} $\calC(S)$.  
Let $\calC^{(0)}(S)$ be the set of proper isotopy classes of essential, non-peripheral curves in $S$.  
We have an edge $(\alpha, \beta) \in \calC^{(1)}$ exactly when $\alpha$ and $\beta$ are distinct and $i(\alpha, \beta) = 0$.  

We define the \emph{arc graph} $\calA(S)$ similarly:  
$\calA^{(0)}(S)$ is the set of proper isotopy classes of essential arcs in $S$.  
Again we have an edge $(\alpha, \beta) \in \calA^{(1)}$ exactly when $\alpha$ and $\beta$ are distinct and $i(\alpha, \beta) = 0$.  
Note that $\calA(S)$ is empty when $S$ is closed.

Masur and Schleimer generalise the definition of the arc graph slightly, as follows~\cite[Definition~7.1]{ms}.  
Suppose that $\Delta \subset \bdy S$ is a non-empty collection of boundary components.
We define $\calA(S, \Delta)$ to be the sub-graph of $\calA(S)$ spanned by the arcs having both endpoints in $\Delta$.
Note that $\calA(S, \bdy S) = \calA(S)$.

We next define the \emph{arc and curve graph} $\AC(S)$:
the zero-skeleton is exactly $\calA^{(0)}(S) \cup \calC^{(0)}(S)$.
Edges come from having disjoint representatives, as before. 
Note that the inclusion of $\calC^{(0)}(S)$ into $\AC^{(0)}(S)$ induces a quasi-isometry of graphs.

The definition of the curve complex must be modified when $\xi(S) \leq 1$: for $S_{1,1}$ we use $i(\alpha, \beta) = 1$ and for $S_{0,4}$ we use $i(\alpha, \beta) = 2$.
For both of these surfaces the graph of curves is a copy of the \emph{Farey graph}.
When $S$ is an annulus we define $\calC^{(0)}(S)$ to be the set of proper isotopy classes of essential properly emebedded arcs where now isotopies are required to fix boundary points. 
Two classes span an edge if they have representatives which are disjoint on their interiors.

All of the various curve, arc, and arc and curve graphs are connected when they are non-empty~\cite[Lemma~2.1]{mm}.
We make each of these into a metric graph by decreeing that all edges have length one.
It is then a theorem of Masur and Minsky that, for any surface $S$ with $\xi(S) \geq 1$, the curve complex $\calC(S)$ is Gromov hyperbolic~\cite[Theorem~1.1]{mm}.  
Masur and Schleimer proved that the same holds for $\calA(S, \Delta)$~\cite[Theorem~20.3]{ms}.

\subsection{\texorpdfstring{$I$}{I}--bundles}

Suppose that $F$ is a connected compact surface, possibly non-orientable, with non-empty boundary.  
Let $\rho \from T \to F$ be an $I$--bundle.   
We call $F$ the \emph{base surface} of the bundle.  
We define $\bdy_v T = \rho^{-1}(\bdy F)$ to be the \emph{vertical boundary} of $T$.  
We define the closure 
\[
\bdy_h T = \closure{\bdy T - \bdy_v T}
\]
to be the \emph{horizontal boundary} of $T$.  
Also, we define the curves 
\[
\bdy (\bdy_h T) = \bdy (\bdy_v T)
\]
to be the \emph{corners} of $T$.  
Finally, there is an involution $\tau \from \bdy_h T \to \bdy_h T$ associated to $\rho$ obtained by swapping the ends of interval fibres.

We now define $\rho_F \from T_F \to F$ to be the \emph{orientation} $I$--bundle over $F$.
Here the preimage under $\rho_F$ of a simple closed curve $\alpha$ is an annulus or a M\"obius band as $\alpha$ is or is not, respectively, orientation preserving in $F$.  
When $F$ is non-orientable we call $T_F$ \emph{twisted} and so $\bdy_h T_F / \tau \homeo F$ is non-orientable.  
If $T_F$ is not twisted then $T_F \homeo F \cross [-1, 1]$ is a product.  
In this case $\tau | (F \cross \{-1\})$ is a homeomorphism from $F \cross \{-1\}$ to $F \cross \{1\}$.

\subsection{Compression bodies}
\label{Sec:CompressionBodies}

References on compression bodies, of the type we are interested in here, include~\cite[Appendix~B]{Bonahon83} and~\cite[Section~1]{Oertel02}.

Suppose that $S = S_{g,b}$ is a surface.  
We assume that $S$ is neither a disk nor a sphere.  
We form $T = S \cross I$.  
We take $\bdy^+ T = S \cross \{1\}$ and $\bdy^- T = S \cross \{0\}$ to be the \emph{upper} and \emph{lower} boundaries of $T$.  
As before, $\bdy_v T = \bdy S \cross I$ is the \emph{vertical boundary}.  
We now attach a collection of three-dimensional two- and three-handles to the lower boundary of $T$ to obtain a three-manifold $M$.  
We define $\bdy^+\! M = \bdy^+ T$ as well as $\bdy_v M = \bdy_v T$.  
Finally, we define 
\[
\bdy^-\! M = \closure{\bdy M - (\bdy^+\! M \cup \bdy_v M)}
\]
Thus $M$ is a \emph{compression body}.  
If $\bdy^-\! M$ is not homeomorphic to $\bdy^+\! M$, then $M$ is \emph{non-trivial}.  
If $\bdy^-\! M$ has no sphere or disk components, then $M$ is \emph{spotless}.  
To simplify the notation we take $S = \bdy^+\! M$.  
Note that if $\bdy^-\! M$ is empty then $M$ is necessarily a handlebody of positive genus. 

We now state the classification of compression bodies.

\begin{theorem}
\label{Thm:ClassificationOfCompBodies}
Suppose that $M$ and $N$ are compression bodies.  
Then $(M, \bdy_v M)$ is homeomorphic to $(N, \bdy_v N)$ if and only if $(\bdy^+ M, \bdy^-\! M)$ is homeomorphic to $(\bdy^+ N, \bdy^-\! N)$. \qed
\end{theorem}

The proof is similar to that of the classification of surfaces~\cite[Theorem~1.1]{FarbMargalit12} and of handlebodies~\cite[Theorem~2.2]{Hempel76}. 
The case of $\bdy_v M = \emptyset$ is discussed by Biringer and Vlamis~\cite[Corollary~2.3]{BiringerVlamis17}.

\subsection{Disk graphs}

Suppose that $(M, S)$ is a non-trivial, spotless compression body.  
Suppose that $(D, \bdy D) \subset (M, S)$ is a properly embedded disk.  
We call $D$ \emph{essential} if $\bdy D$ is essential in $S$.  
We now define the \emph{disk graph} $\calD(M, S)$.  
The vertices of $\calD(M, S)$ are proper isotopy classes of essential disks in $(M, S)$.  
A pair of distinct vertices $D$ and $E$ give an edge $(D, E) \in  \calD(M, S)$ if they have disjoint representatives.  

\subsection{Subsurface projection}

We give one of the standard definitions of subsurface projection~\cite[Definition~4.4]{ms}. 

\begin{definition}
\label{Def:SubSurfaceProjection}
Suppose that $X$ is an essential subsurface, but not a pair of pants, in $S$.  
The relation of \emph{subsurface projection}, $\pi_X \from \AC(S) \to \calC(X)$, is defined as follows.  
Let $\rho_X \from S^X \to S$ be the covering map corresponding to the subgroup $\pi_1(X) < \pi_1(S)$.  
Note that the Gromov compactification of $S^X$ is homeomorphic to $X$.  
This gives an identification of the graphs $\calC(S^X)$ and $\calC(X)$.  
For any $\alpha \in \AC(S)$ we define $\alpha^X = \rho_X^{-1}(\alpha)$ to be the full preimage.  
We define $\alpha|X$ to be the essential arcs, and essential non-peripheral curves, of $\alpha^X$.  
If $X$ is an annulus, then we set $\pi_X(\alpha) = \alpha|X$.  
Otherwise, for every $\beta \in \alpha|X$ we form $N = N(\beta \cup \bdy X)$ and we place the essential isotopy classes of $\bdy_X N$ into $\pi_X(\alpha)$.  
\end{definition}

Note that if $\alpha$ misses $X$ then $\pi_X(\alpha)$ is empty.
Suppose instead that $\alpha$ cuts $X$.
If $X$ is an annulus, then the diameter of $\pi_X(\alpha)$ is at most one.
If $X$ is not an annulus, then the diameter of $\pi_X(\alpha)$ is at most two~\cite[Lemma~2.3]{mm2}.
If $X$ and $Y$ are essential subsurfaces of $S$ then we define $\pi_Y(X) = \pi_Y(\bdy_S X)$.  
If $X$ and $Y$ are disjoint, or if $Y$ is nested in $X$, then this is empty.
We record the following for later use. 

\begin{lemma}
\label{Lem:Overlap}
Suppose that $X$ and $Y$ are overlapping essential subsurfaces of $S$.
Then $\pi_Y(X)$ is non-empty and has diameter at most two. \qed
\end{lemma}

We will adopt the following useful shorthand notation.  
Suppose that $\alpha$ and $\beta$ are curves or arcs, both cutting an essential subsurface $X \subset S$.  
Then 
\[
d_X(\alpha, \beta) = d_{\calC(X)}(\pi_X(\alpha), \pi_X(\beta))
\]
is the \emph{subsurface projection distance} between $\alpha$ and $\beta$ in $X$.

\section{Bound for the arc graph}

Let $S = S_{g,b}$, where $\xi'(S) > 0$ and $b > 0$.  
Take $\Delta \subset \bdy S$ to be a non-empty union of components.   
Let $\calA(S, \Delta)$ be the graph of essential arcs with endpoints in $\Delta$. 

\subsection{Witnesses for the arc graph}

\begin{definition}
An essential subsurface $X \subset S$ is a \emph{witness} for $\calA(S, \Delta)$ if every arc $\alpha \in \calA(S, \Delta)$ cuts $X$. 
\end{definition}

We repackage a few results~\cite[Lemmas~5.9 and~7.2]{ms}.

\begin{lemma}
\label{Lem:ArcWitness}
Suppose that $X \subset S$ is an essential subsurface, but not an annulus or a pair of pants.  
The following are equivalent.
\begin{itemize}
\item
$X$ is a witness for $\calA(S, \Delta)$. 
\item
$X$ contains $\Delta$. 
\item
For all arcs $\alpha \in \calA(S, \Delta)$, the projection $\pi_X(\alpha)$ is non-empty. 
\item
The projection $\pi_X \from \calA(S, \Delta) \to \calC(X)$ is coarsely Lipschitz with a constant of $2$. \qed
\end{itemize}
\end{lemma}


We have a useful corollary.

\begin{corollary}
\label{Cor:ArcWitnessesOverlap}
Suppose that $X$ and $Y$ are distinct witnesses with $\xi'(X) = \xi'(Y)$.  
Then $X$ and $Y$ overlap.
\end{corollary}

\begin{proof}
By \reflem{ArcWitness} both $X$ and $Y$ contain $\Delta$;
thus they intersect.
Since they have the same modified complexity, by \refrem{Nested} they cannot be nested.
Thus they overlap. 
\end{proof}

We let $\MCG(S, \Delta)$ be the mapping class group for the pair $(S, \Delta)$:
the group of mapping classes that preserve $\Delta$ setwise.  
We say that a pair of arcs $\alpha, \beta \in \calA(S, \Delta)$ have the same \emph{topological type} (or more simply, the same \emph{type}) if there is a mapping class $f \in \MCG(S, \Delta)$ so that $f(\alpha) = \beta$.   

\begin{lemma}
\label{Lem:ArcDiamBound}
The quotient $\calA(S, \Delta) / \MCG(S, \Delta)$ has diameter at most two. 
\end{lemma}



\begin{proof}
Suppose that $\alpha \in \calA(S, \Delta)$ is an arc. 
We break the proof into cases, depending on the number of components of $\bdy S$ and of $\Delta$.

Suppose that $\Delta$ has at least two components.
Then $\alpha$ is disjoint from some $\gamma \in \calA(S, \Delta)$ meeting two components of $\Delta$.  
Since we made no assumptions on $\alpha$, we find that in the quotient, all vertices are at most distance one from $[\gamma]$.  
Thus the diameter is at most two. 

Suppose that $\Delta$ has only one component, but $\bdy S$ has at least two. 
Let $\delta$ be some component of $\bdy S - \Delta$.
Then $\alpha$ is disjoint from some arc $\gamma \in \calA(S, \Delta)$ so that $\gamma$ separates $\delta$ from the rest of $S$.
We again obtain a diameter bound of two. 

Suppose that $\bdy S$ has a single component, which necessarily equals $\Delta$. Then $\alpha$ is disjoint from some arc $\gamma \in \calA(S, \Delta)$ so that $\gamma$ is non-separating.  
This gives the diameter bound and finishes the proof. 
\end{proof}

\subsection{Families of witnesses}
\label{Sec:ArcFamilies}

Fix a number $c \leq \xi'(S)$. 
The collection 
\[
\calF_c = \{ X \subset S \st \mbox{$X$ is a witness for $\calA(S, \Delta)$ and $\xi'(X) = c$} \}
\]
is called a \emph{family}.  
We only consider non-empty families.  

Suppose that $X, Y, Z \in \calF_c$ are witnesses, with $Y$ distinct from both $X$ and $Z$.  Then we define 
\[
d_Y(X,Z) = d_{\calC(Y)}(\pi_Y(X), \pi_Y(Z))
\]
Note that this is well-defined by \refcor{ArcWitnessesOverlap}.  
We note that there is an abuse of notation here: 
the family $\calF_c$ consists of witnesses $X$ -- that is, surfaces -- not metric graphs.
However each witness $X$ gives a metric graph, namely $\calC(X)$. 
We trust this will not cause confusion. 

\begin{lemma}
\label{Lem:ArcAxioms}
For every $c$ the family $\calF_c$ satisfies the three BBF axioms given in \refsec{BBF}.
Since there are only finitely many of these families, there is a common constant $\Link$ that works for all of them simultaneously. 
\end{lemma}

\begin{proof}
\refax{BoundedProjections} follows from \refcor{ArcWitnessesOverlap} and \reflem{Overlap}.

\refax{Behrstock} follows from \refcor{ArcWitnessesOverlap} and the usual Behrstock inequality~\cite[Theorem~4.3]{Behrstock06}.  
See~\cite[Lemma 2.5]{Ma} for an elementary proof following ideas of Leininger.

\refax{LargeLink} appears as Lemma~6.2 in~\cite{mm2}. 
See \cite[Lemma~5.3]{bbf} for a proof giving a concrete bound and avoiding the machinery of hierarchies.
\end{proof}

We now apply the BBF construction, outlined in \refsec{BBF}, to each family $\calF_c$.  
This gives us a quasi-tree of curve graphs $\calC(\calF_c)$. 
We deduce that $\calC(\calF_c)$ is a hyperbolic metric graph where each of the curve complexes $\calC(X)$, for $X \in \calF_c$, embeds as a totally geodesic subgraph.  
From~\cite[Corollary~1.1]{bb} and \refthm{BBFBound} we deduce that 
\[
\asydim \calC(\calF_c) \leq c + 1 
\]
If $c = 1$ then \refthm{BBFBoundTree} allows us to sharpen the bound:
\[
\asydim \calC(\calF_1) \leq 1 
\]
On the other hand, if $c = 4g + b - 3$ then $\calF_c = \{S\}$ and we have
\[
\asydim \calC(\calF_c) \leq 4g + b - 3
\]
We now define $\calP(S, \Delta)$ to be the product, equipped with the $\ell^1$ metric, of the quasi-trees of curve graphs $\calC(\calF_c)$ as $c$ ranges from one to $\xi'(S)$.
From the above and from \reffac{Product} we deduce the following. 

\begin{corollary}
\[
\pushQED{\qed}
\asydim \calP(S, \Delta) \leq \frac{(4g+b) (4g+b-3)}{2} - 2 \qedhere
\popQED
\]
\end{corollary}


\subsection{Embedding the arc graph}
\label{Sec:ArcEmbed}

In this section we fix the constants $\Link$, $\Cutoff$, and $\Close$.
We then state and prove \refthm{ArcEmbed}.  

The constant $\Link$ is the larger of $13$ (as explained in the proof of \reflem{ArcCoarseLip}) and the constant given by \reflem{ArcAxioms}.
The constant $\Cutoff$ is now given by~\cite[Theorem~A]{bbf}.
Finally, the constant $\Close$ is provided by \reffac{Close}.
We take $\calP(S, \Delta)$ to be the product of the resulting quasi-trees. 
Here is the statement. 

\begin{theorem}
\label{Thm:ArcEmbed}
There is a quasi-isometric embedding $\phi$ of the arc graph $\calA(S, \Delta)$ into the product $\calP(S, \Delta)$ of quasi-trees of curve graphs.  
Moreover, $\phi$ is equivariant with respect to the action of the mapping class group $\MCG(S, \Delta)$. 
\end{theorem}


From this, and from \reffac{Subspaces}, we deduce the following.

\begin{corollary}
\label{Cor:ArcBound}
Suppose that $S = S_{g,b}$ has non-empty boundary.  
Suppose that $\Delta \subset \bdy S$ is a non-empty union of components.  
Finally, suppose that $\xi'(S) \geq 1$.  
Then 
\[
\pushQED{\qed}
\asydim \calA(S, \Delta) \leq \frac{(4g+b) (4g+b-3)}{2} - 2 \qedhere
\popQED
\]
\end{corollary}


We now turn to the proof of \refthm{ArcEmbed}.  
Fix a modified complexity~$c$.  

\begin{definition}
\label{Def:ArcCarries}
Suppose that $\beta \in \calA(S, \Delta)$ is an arc.
Suppose that $Y \in \calF_c$ is a witness, and $\beta$ has a representative contained in $Y$.
Then we say that $Y$ \emph{carries} $\beta$.  
\end{definition}

Note that $\pi_Y(\beta) \subset \calC(Y)$ is one or two essential, non-peripheral curves in $Y$.  
Recall, by \reffac{Embed}, that $\calC(Y)$ embeds into $\calC(\calF_c)$. 
We now define a relation $\phi_c \from \calA(S, \Delta) \to \calC(\calF_c)$ as follows:
\[
\phi_c(\alpha) = \{ \pi_Y(\beta) \st \mbox{$d_\calA(\alpha, \beta) \leq 2$ and $Y \in \calF_c$ carries $\beta$} \}
\]

\begin{lemma}
\label{Lem:ArcCoarseLip}
The relation $\phi_c$ is an equivariant coarse Lipschitz map.
\end{lemma}


\begin{proof}
Equivariance follows from the definition.  

The set $\phi_c(\alpha)$ is non-empty by \reflem{ArcDiamBound}.  
Suppose that $Y$, $Y'$, and $Z$ lie in $\calF_c$.  
Suppose that $\beta$ and $\beta'$ are carried by $Y$ and $Y'$, respectively, and are distance at most two from $\alpha$. 
Thus $d_\calA(\beta, \beta') \leq 4$.
We deduce that $d_Z(Y, Y')$ is at most twelve.

We now recall our choices (made above) of $\Link$, $\Cutoff$, and $\Close$. 
In particular we have $\Link > 12$. 
Thus, by \reffac{Close}, we have that $\pi_{Y}(\bdy Y')$ and $\pi_{Y'}(\bdy Y)$ are distance at most $\Close$ in $\calC(\calF_c)$.
Thus $d_\calC(\beta, \beta') \leq \Close + 20$, bounding the diameter of $\phi_c(\alpha)$.
Thus $\phi_c$ is a coarse map.

Furthermore, by \cite[Lemma~2.3]{mm2}, if $d_\calA(\alpha, \alpha') = 1$ then the distance between $\phi_c(\alpha)$ and $\phi_c(\alpha')$ is also bounded in terms of $\Close$.  
Applying the triangle inequality gives the result. 
\end{proof}

\begin{lemma}
\label{Lem:ArcNoMove}
Suppose that $\alpha, \gamma \in \calA(S, \Delta)$ are arcs and $X$ is a witness with $\xi'(X) = c$.  Then 
\[
|d_X(\alpha, \gamma) - d_X(\phi_c(\alpha), \phi_c(\gamma))| \leq 12
\]
\end{lemma}

\begin{proof}
Suppose that $\beta \in \calA(S, \Delta)$ has $d_\calA(\alpha, \beta) \leq 2$ and $\beta$ is carried by some witness $Y$ with $\xi'(Y) = c$.  
Note that $\alpha$ and $\pi_Y(\beta)$ are distance at most three in $\AC(S)$, the arc and curve complex for $S$.  

Now, if $X = Y$ then we have
\[
d_X(\alpha, \pi_Y(\beta)) 
   = d_X(\alpha, \pi_X(\beta)) 
   = d_X(\alpha, \beta)
   \leq 4
\]
by \cite[Lemma~2.3]{mm2}.  
If $X \neq Y$ then instead we have
\[
d_X(\alpha, \pi_Y(\beta)) 
   = d_{\calC(X)}(\pi_X(\alpha), \pi_X(Y))
   \leq 6
\]
This holds for all $\beta$ arising in the definition of $\phi_c(\alpha)$.  The lemma now follows by applying the triangle inequality twice.
\end{proof}



We now define $\phi \from \calA(S, \Delta) \to \calP(S, \Delta)$ by taking
\[
\phi(\alpha) = (\phi_c(\alpha))_c
\]
All that remains is to prove that $\phi$ is a quasi-isometric embedding.  
Suppose that $\alpha$ and $\gamma$ are arcs in $\calA(S, \Delta)$.  
We must show that $d_\calA(\alpha, \gamma)$ and $d_\calP(\phi(\alpha), \phi(\gamma))$ are coarsely equal. 

We first bound $d_\calP(\phi(\alpha), \phi(\gamma))$ from above. 
Recall that $\calP(S, \Delta)$ is equipped with the $\ell^1$ metric and so 
\[
d_\calP(\phi(\alpha), \phi(\gamma))
  = \sum_c d_{\calC(\calF_c)}(\phi_c(\alpha), \phi_c(\gamma))
\]
Each of the terms on the right-hand side is bounded in terms of $d_\calA(\alpha, \gamma)$ by \reflem{ArcCoarseLip} and we are done. 

We now bound $d_\calA(\alpha, \gamma)$ from above. 
Since $\calA(S, \Delta)$ is a \emph{combinatorial complex}, in the sense of~\cite[Section~5]{ms}, we have a corollary of~\cite[Theorems~5.14 and~13.1]{ms}.

\begin{theorem}
\label{Thm:ArcEstimate}
Suppose that $S$ and $\Delta$ are as above.
There is a constant $\mathsf{L}$ so that for any $\mathsf{L}' \geq \mathsf{L}$ there is a constant $\Coarse$ with the following property.
For any arcs $\alpha$ and $\gamma$ we have
\[
d_\calA(\alpha, \gamma) =_\Coarse \sum [d_X(\alpha, \gamma)]_{\mathsf{L}'}
\]
where the sum is taken over all witnesses $X$ for $\calA(S, \Delta)$. \qed
\end{theorem}

Take $\Cutoff' > 12$ sufficiently larger than the constants $\Cutoff$ and $\mathsf{L}$ appearing in Theorems~\ref{Thm:BBFEstimate} and~\ref{Thm:ArcEstimate}, respectively.
Set $\mathsf{L}' = \Cutoff' + 12$.  
Fix a witness $X$ and set $c = \xi'(X)$.   
If a term $d_X(\alpha, \gamma)$ appears in the upper bound of \refthm{ArcEstimate} then, by \reflem{ArcNoMove}, the term $d_X(\phi_c(\alpha), \phi_c(\gamma))$ appears in the lower bound provided by \refthm{BBFEstimate} for the family $\calF_c$.
Also, $d_X(\alpha, \gamma)$ is at most twice $d_X(\phi_c(\alpha), \phi_c(\gamma))$ (by \reflem{ArcNoMove} and because $\Cutoff' > 12$).
Thus $d_\calA(\alpha, \gamma)$ is coarsely bounded above by 
$d_\calP(\phi(\alpha), \phi(\gamma))$, as desired.
This finishes the proof of \refthm{ArcEmbed}. \qed




\section{Bound for the disk complex}
\label{Sec:Disk}

Suppose that $M$ is a spotless compression body, as defined in \refsec{CompressionBodies}.
Suppose that $S = S_{g,b} = \bdy^+\!M$ is the upper boundary.  
We assume that $\xi'(S) > 0$.  
Let $\calD(M, S)$ be the graph of essential disks with boundary in $S$. 

\subsection{Witnesses for the disk complex}

\begin{definition}
An essential subsurface $X \subset S$ is a \emph{witness} for $\calD(M, S)$ if every disk $D \in \calD(M, S)$ cuts $X$. 
\end{definition}

Some authors call such an $X$ \emph{disk-busting}~\cite{ms}. 
We call a witness $X$ \emph{large} if it satisfies
\[
\diam_X(\pi_X(\calD(M, S))) > 60
\]
We now record the classification of large witnesses~\cite[Theorems~10.1, 11.10, 12.1]{ms}.

\begin{theorem}
\label{Thm:DiskWitness}
Suppose that $(M, S)$ is a non-trivial spotless compression body.
Suppose that $X \subset S$ is a large witness for $\calD(M, S)$.
Then we have the following. 
\begin{itemize}
\item
$X$ is not an annulus. 
\item
If $X$ compresses in $M$, then there are disks $D$ and $E$ with boundary contained in, and filling, $X$. 
\item
If $X$ is incompressible in $M$, then there is an orientation $I$--bundle $\rho_F \from T_F \to F$ with $(T_F, \bdy_h T_F) \subset (M, S)$ and with $X$ being a component of $\bdy_h T_F$.  
Also $\bdy_v T_F$ is properly embedded in $(M, S)$ and at least one component of $\bdy_v T_F$ is isotopic into $S$.  Also, $F$ admits a pseudo-Anosov map. \qed
\end{itemize}
\end{theorem}

\begin{remark}
\label{Rem:Paring}
Suppose that $X$ is a large incompressible witness, as in the third case.
Let $F$ be the base of the associated $I$--bundle $T_F$.  
Let $P_F$ be the collection of annuli, embedded in $S$, which are isotopic, rel boundary, to components of $\bdy_v T_F$.
We call these annuli the \emph{paring locus} for $T_F$.
The paring locus $P_F$ is non-empty (\refthm{DiskWitness}) and is disk busting~\cite[Remark~12.17]{ms}.
We call an essential disk $(D, \bdy D) \subset (M, S)$ \emph{vertical} for $T_F$ exactly when $D \cap T_F$ is vertical in $T_F$.
Such disks exist by \refthm{DiskWitness}.
If $D$ is vertical for $T_F$ then $D$ meets the paring locus $P_F$ in exactly two essential arcs.
Finally, the union $P_F \cup \bdy_h T_F$ is a compressible witness for $\calD(M, S)$. 
\end{remark}

\begin{definition}
\label{Def:Twins}
If $T_F$ is a product then $\bdy_h T_F = X \sqcup X'$ where $X' \subset S$ is again a large incompressible witness.  
In this case we call $X$ and $X'$ \emph{twins}.
\end{definition}

Similar to the case of the arc complex (\refcor{ArcWitnessesOverlap}), all large witnesses for the disk complex \emph{interfere}, as follows.

\begin{corollary}
\label{Cor:DiskWitnessesInterfere}
Suppose that $X$ and $Y$ are disjoint large witnesses with $\xi'(X) = \xi'(Y)$.  
Then either 
\begin{enumerate}
\item
$X$ and $Y$ are twins or
\item
$X$ and $Y$ have twins, $X'$ and $Y'$ respectively, so that 
$X'$ overlaps with $Y$ and $Y'$ overlaps with $X$. 
\end{enumerate}
\end{corollary}


\begin{proof}
Suppose that $X$ and $Y$ are not twins. 
Thus we may apply~\cite[Lemma~12.21]{ms} to find that $X$ has a twin $X'$ and $X'$ intersects $Y$. 
Since $\xi'(X') = \xi'(Y)$, neither $X'$ nor $Y$ is nested in the other.
Thus $X'$ overlaps with $Y$.  
The remainder of the proof is similar. 
\end{proof}

As usual we take $\MCG(M, S)$ to be the mapping class group for the pair $(M, S)$.  
That is, the group of mapping classes of $M$ that preserve $S$ setwise.  
We say that a pair of disks $D, E \in \calD(M, S)$ have the same \emph{topological type} (or more simply, the same \emph{type}) if there is a mapping class $f \in \MCG(M, S)$ so that $f(D) = E$.   

\begin{lemma}
\label{Lem:DiskDiamBound}
The quotient $\calD(M, S) / \MCG(M, S)$ has diameter at most two. 
\end{lemma}


\begin{proof}
There are two cases.  
Suppose first that $\calD(M, S)$ contains a non-separating disk: 
that is, a disk $(D, \bdy D) \subset (M, S)$ so that $M - D$ is connected. 
By \refthm{ClassificationOfCompBodies} (the classification of compression bodies) all non-separating disks have the same topological type. 
Suppose that $E \in \calD(M, S)$ is a separating disk.  
The classification of compression bodies implies that $E$ is disjoint from some non-separating disk $D'$.
This gives the desired diameter bound. 

Suppose instead that all essential disks in $(M, S)$ are separating.
Thus the lower boundary of $M$ is non-empty.
Suppose that $F$ is a component of the lower boundary of $M$.
Let $D$ be a separating disk which cuts a copy of $F \cross I$ off of $M$.
Now suppose that $E$ is any separating disk.
The classification of compression bodies implies that $E$ is disjoint from some disk $D'$ which is a homeomorphic image of $D$.
\end{proof}

\subsection{Families of witnesses}

Fix a modified complexity $c \leq \xi'(S)$.
The collection
\[
\calF_c = \{ X \subset S \st \mbox{$X$ is a large witness for $\calD(M, S)$ and $\xi'(X) = c$} \}
\]
is called a \emph{complete family}.
We now define a \emph{reduced family} as follows.
Suppose that $X$ and $X'$ in $\calF_c$ are twins.
Thus there is an $I$--bundle $T = F_T \cross I$ where $\bdy_h T = X \sqcup X'$.
We remove both $X$ and $X'$ from $\calF_c$ and replace them by the base surface $F_T$.
We abuse notation and again use $\calF_c$ to denote the reduced family.
When $A = F_T \in \calF_c$ is the base surface associated to $T$ we abuse notation and define $\bdy_S A = \bdy_S \bdy_h T$.


\begin{lemma}
\label{Lem:CurvesCut}
Fix $(M, S)$ and $c$.  
Suppose that $A, B \in \calF_c$ are distinct.
Suppose that $B$ is a base surface replacing the twinned surfaces $Y$ and $Y'$.
Then $\bdy_S A$ cuts both $Y$ and $Y'$. 
\end{lemma}


\begin{proof}
Let $T_B$ be the product $I$--bundle associated to $B$.
Let $P_B$ be the paring locus of $T_B$.
Recall that $P_B$ is disk busting.

Suppose that $A$ is a compressible witness.
Thus $P_B$ cuts $A$.
Suppose that $\bdy_S A$ does not cut $Y$.  
Thus $Y$ is (after an isotopy) contained in $A$. 
From \refrem{Nested} we deduce that $Y$ is isotopic to $A$. 
Thus $P_B$ does not cut $A$, a contradiction. 
Thus $\bdy_S A$ cuts $Y$; a similar argument proves that $\bdy_S A$ cuts $Y'$. 

Suppose that $A$ is an incompressible witness, with $I$--bundle $T_A$. 
Let $P_A$ be the paring locus of $T_A$.  
There are two subcases as $T_A$ is twisted or a product. 

Suppose that $T_A$ is twisted.  
Thus $A \cup P_A$ is a compressible witness (\refrem{Paring}). 
Again, since $P_B$ is disk busting, it cuts $A \cup P_A$.
If $P_B$ cuts $\bdy_S A$ we are done because $P_B$ is parallel into both $Y$ and $Y'$. 
If not, then $P_B$ is (after an isotopy) contained in $A$ or contained in $P_A$. 
In either case $Y$ and $Y'$ must cut $A$. 
Since $\xi'(Y) = \xi'(Y') = \xi'(A)$ we cannot have $Y$ or $Y'$ contained in $A$ (\refrem{Nested}). 
Thus $A$ overlaps \emph{both} $Y$ and $Y'$, and we are done. 

Suppose that $T_A$ is a product.  
Let $X$ and $X'$ be the twin components of $\bdy_h T_A$. 
Appealing to \refcor{DiskWitnessesInterfere} we may assume that $X$ overlaps with $Y$.  If $X$ overlaps with $Y'$ then we are done. 
If it does not, then, by \refcor{DiskWitnessesInterfere}, we deduce that $X'$ overlaps $Y'$.  Thus in either case, we are done. 
\end{proof}

\begin{definition}
\label{Def:WitnessProjection}
Fix a modified complexity $c$.
Suppose that $A, B \in \calF_c$.
We now define $\pi_B(A)$. 
(Note that we are overloading the notation $\pi_B$.
When the argument is a collection of curves we use \refdef{SubSurfaceProjection}.
When the argument is a witness we use \refdef{WitnessProjection}.)
There are two cases.
\begin{enumerate}
\item
Suppose that $B$ is not a base surface.
Then we define $\pi_B(A) = \pi_B(\bdy_S A)$.
\item
Suppose that $B$ is a base surface. 
Suppose that $\rho_B \from T_B \to B$ is the $I$--bundle associated to $B$.
Then we isotope $\bdy_S A$ to meet $\bdy_h T_B$ minimally and we define $\pi_B(A) = \pi_B(\rho_B(T_B \cap \bdy_S A))$.
\qedhere
\end{enumerate}
\end{definition}

\noindent
In the above definition, we are considering $\rho_B(T_B \cap \bdy_S A)$ as a set of arcs and curves in $B$.  
We surger them one at a time to obtain a set of curves in $B$.
Also, in both parts of the definition, \refcor{DiskWitnessesInterfere} implies that $\pi_B(A)$ is non-empty.

Suppose that $A, B, C \in \calF_c$.
Suppose further that $B$ is distinct from both $A$ and $C$.
Then we define 
\[
d_B(A,C) = d_{\calC(B)}(\pi_B(A), \pi_B(C))
\]
We will now abuse notation: 
the reduced family $\calF_c$ contains surfaces $A$ which index metric spaces, namely the curve graphs $\calC(A)$, instead of being metric spaces themselves.  

We now verify the three BBF axioms, as stated in \refsec{BBF}.  
Recall that $(M, S)$ is a spotless compression body, together with its upper boundary.  
Also, $c \in \ZZ$ is an integer.
Here is the proof of \refax{BoundedProjections}.

\begin{lemma}
\label{Lem:DiskAxBoundedProjections}
There is a constant $k > 0$ so that for every $(M, S)$, for every $c$, and for every $A, B \in \calF_c$, we have that $\diam_B(\pi_B(A)) \leq k$. 
\end{lemma}

\begin{proof}
Suppose that $B$ is not a base surface.  
Note that $\bdy_S A$ is a disjoint collection of curves.  
By \refcor{DiskWitnessesInterfere}, we have $\bdy_S A$ cuts $B$.
By~\cite[Lemma~2.3]{mm2} the diameter of $\pi_B(A)$ in $\calC(B)$ is at most two. 

Suppose that $B$ is a base surface.  
Let $D$ be either a compressing disk for $A$ or a vertical disk for $T_A$, as provided by \refthm{DiskWitness} or \refrem{Paring}, respectively. 
If $D$ is a compressing disk then $\bdy D$ is disjoint from $\bdy_S A$. 
If $D$ is a vertical disk then $\bdy D$ meets $\bdy_S A = \bdy_S \bdy_h T_A$ in exactly four points. 
We now isotope $\bdy D$ to have minimal intersection with $\bdy_S B$.
Applying \cite[Lemma~12.20]{ms} we deduce that $\pi_B(\rho_B(T_B \cap \bdy D))$ has bounded diameter.  
Thus by the triangle inequality $\pi_B(\rho_B(T_B \cap \bdy_S A))$ also has bounded diameter, and we are done.
\end{proof}

We adopt the notation $d_Y(A, C) = d_Y(\bdy_S A, \bdy_S C)$. 

\begin{lemma}
\label{Lem:Swap}
Let $k$ be the constant of \reflem{DiskAxBoundedProjections}.  
Fix $(M, S)$ and $c$.  
Suppose that $A, B, C \in \calF_c$ where $B$ is a base surface replacing the twinned surfaces $Y, Y'$.  
Then 
\[
d_Y(A, C) \leq d_B(A, C) \leq d_Y(A, C) + 2k
\]
and the same holds for $Y'$.
\end{lemma}

\begin{proof}
Applying \reflem{CurvesCut}, we have that $d_Y(A, C)$ is defined.  
The first inequality follows from the definitions. 
The second inequality follows from two applications of \reflem{DiskAxBoundedProjections} and the triangle inequality. 
\end{proof}

Here is the proof of \refax{Behrstock}.

\begin{lemma}
Let $k$ be the constant of \reflem{DiskAxBoundedProjections}.
Fix $(M, S)$ and $c$.
For every $A, B, C \in \calF_c$, we have that at most one of the following is greater than $12 + 2k$:
\[
d_A(B, C), \quad 
d_B(A, C), \quad 
d_C(A, B)
\]
\end{lemma}

\begin{proof}
It suffices to assume that $d_B(A, C) > 12 + 2k$ and bound $d_A(B, C)$ from above.

Suppose that neither $B$ nor $A$ is a base surface.  
Then we may apply the usual Behrstock inequality and deduce that $d_A(B, C) < 10$; see~\cite[Lemma 2.5]{Ma}.

Suppose instead $B$ is not a base surface, but $A$ is.
Let $X$ and $X'$ be the twins over $A$. 
By \refcor{DiskWitnessesInterfere}, both $X$ and $X'$ overlap $B$. 
Applying~\cite[Lemma~2.3]{mm2} we have $d_B(X, C) > 10 + 2k$. 
The usual Behrstock inequality gives $d_X(B, C) < 10$. 
Applying \reflem{Swap} we deduce that $d_A(B, C) < 10 + 2k$. 

Suppose instead that $B$ is a base surface but $A$ is not.
Let $Y$ and $Y'$ be the twins over $B$.  
By \reflem{Swap} we have that $d_Y(A, C) \geq 12$ and also $d_{Y'}(A, C) \geq 12$.
By \reflem{CurvesCut} both $\bdy_S B$ and $\bdy_S C$ cut $A$.
Suppose that $\bdy_S Y$ cuts $A$. 
The usual Behrstock inequality gives $d_A(Y, C) < 10$. 
We deduce that $d_A(B, C) < 12$, as desired.

Finally suppose that $B$ and $A$ are base surfaces. 
Let $Y$ and $Y'$, and $X$ and $X'$, be the twins over $B$ and $A$, respectively. 
By \refcor{DiskWitnessesInterfere}, we may assume that $X$ and $Y$ overlap. 
Thus $d_Y(X, C) > 10$, so $d_X(Y, C) < 10$, and we are done as above.
\end{proof}

Here is the proof of \refax{LargeLink}.

\begin{lemma}
Let $k$ be the constant of \reflem{DiskAxBoundedProjections}.  Fix $(M, S)$ and $c$.  For every $A, C \in \calF_c$, the following set is finite:
\[
\left\{ B \in \calF_c \st \mbox{$A \neq B$, $B \neq C$, and $d_B(A, C) > 7 + 2k$} \right\}
\]
\end{lemma}

\begin{proof}
If $A$ and $C$ are not base surfaces, then this follows from~\cite[Lemma~5.3]{bbf} and \reflem{Swap}.
If $A$ is a base surface but $C$ is not then suppose that $X$ and $X'$ are the twins over $A$.  
We may repeat the previous argument for the pairs $(X, C)$ as well as $(X', C)$, paying at most an additional two~\cite[Lemma~2.3]{mm2} in each case.
When both $A$ and $C$ are base surfaces there are four such pairs and the cost is at most an additional four in each case.
\end{proof}

Since the axioms hold, as in \refsec{ArcFamilies} we may build the product of quasi-trees of spaces $\calP(M, S)$ for the disk graph.  
We obtain the following. 

\begin{corollary}
Suppose that $(M, S)$ is a nontrivial spotless compression body with $S = S_{g,b}$.  
Suppose that $\xi'(S) \geq 1$.  
Then 
\[
\pushQED{\qed}
\asydim \calP(M, S) \leq \frac{(4g+b) (4g+b-3)}{2} - 2 \qedhere
\popQED
\]
\end{corollary}

\subsection{Embedding the disk complex}
\label{Sec:DiskEmbed}

In this section we prove the following. 

\begin{theorem}
\label{Thm:DiskEmbed}
There is a quasi-isometric embedding $\phi$ of the disk graph $\calD(M, S)$ into the product $\calP(M, S)$ of quasi-trees of curve graphs.  
Moreover, $\phi$ is equivariant with respect to the action of the mapping class group $\MCG(M, S)$. 
\end{theorem}

We deduce from this, and from \reffac{Subspaces}, the following.

\begin{corollary}
\label{Cor:DiskBound}
Suppose that $(M, S)$ is a nontrivial spotless compression body with $S = S_{g,b}$.  
Suppose that $\xi'(S) \geq 1$.  
Then 
\[
\pushQED{\qed}
\asydim \calD(M, S) \leq \frac{(4g+b) (4g+b-3)}{2} - 2 \qedhere
\popQED
\]
\end{corollary}

\begin{remark}
When $g > 1$ and $b = 0$, the upper bound is smaller by one.
See \cite[Corollary~1.1]{bb}.
\end{remark}

\begin{remark}
\label{Rem:Hamenstadt}
Hamenst\"adt~\cite[Theorem~3.6]{hii} previously showed, in the special case of a handlebody $H_g$, that $\asydim \calD(H_g, \bdy H_g) \leq (3g - 3)(6g - 2)$.
Her proof technique is quite different from ours. 
\end{remark}

The proof of \refthm{DiskEmbed} is the same as that of \refthm{ArcEmbed}, with three changes.
First, we replace the diameter bound (\reflem{ArcDiamBound}, for arcs) by \reflem{DiskDiamBound}.
Second, we replace the definition of \emph{carries} (\refdef{ArcCarries}, for arcs) with the following. 

\begin{definition}
Suppose that $D \in \calD(M, S)$ is a disk and $A \in \calF_c$.  
If $A \subset S$ is compressible (so not a base surface) then $A$ \emph{carries} $D$ exactly when $D$ is a compressing disk for $A$.  
If $A$ is a twisted witness, or a base surface, then $A$ \emph{carries} $D$ exactly when $D$ is isotopic to a vertical disk in $T_A$. 
\end{definition}

Third and lastly, we replace the distance estimate \refthm{ArcEstimate} with the following. 

\begin{theorem}
Suppose that $M$ and $S$ are as above.  
There is a constant $\Cutoff$ so that for any $\Cutoff' \geq \Cutoff$ there is a constant $\Coarse$ with the following property.
For any disks $D$ and $E$ we have
\[
d_\calD(D, E) =_\Coarse \sum [d_X(D, E)]_{\Cutoff'}
\]
where the sum is taken over all $X$ in all reduced families for $\calD(M, S)$. 
\end{theorem}

\begin{proof}
The distance estimate~\cite[Theorem~19.9]{ms} bounds $d_\calD(D, E)$ above and below using the sum of projection distances to all witnesses.
That is, the sum there is taken over all $X$ in all complete families for $\calD(M, S)$.
Suppose that $B$ is a base surface for the twins $Y$ and $Y'$.  
By~\cite[Theorem~12.20]{ms} we have that all three of 
\[
d_B(D, E), \quad 
d_Y(D, E), \quad \mbox{and} \quad
d_{Y'}(D, E)
\]
are coarsely equal.  
Here we define $d_B(D, E) = d_B(\pi_B(\bdy D), \pi_B(\bdy E))$ and also $\pi_B(\bdy D) = \pi_B(\rho_B(T_B \cap \bdy D))$ as in \refdef{WitnessProjection}.
This and~\cite[Theorem~19.9]{ms} gives the lower bound. 
The upper bound is proved in the same way, after weakening the constant $\Coarse$ by a factor of two.
\end{proof}

\end{document}